\documentclass{amsproc}
\usepackage{amsmath}
\usepackage{amssymb}
\usepackage[colorlinks=true]{hyperref}
\hypersetup{urlcolor=blue, citecolor=red}

\newtheorem{theorem}{Theorem}[section]
\newtheorem{corollary}{Corollary}

\newtheorem{lemma}[theorem]{Lemma}

\newtheorem{conjecture}{Conjecture}

\theoremstyle{definition}
\newtheorem{definition}[theorem]{Definition}
\newtheorem{remark}{Remark}

\newcommand{\be}{\begin{equation}}
\newcommand{\beq}{\begin{equation}}
\newcommand{\eeq}{\end{equation}}
\newcommand{\beqa}{\begin{eqnarray}}
\newcommand{\eeqa}{\end{eqnarray}}
\newcommand{\beqanl}{\begin{eqnarray*}}
\newcommand{\eeqanl}{\end{eqnarray*}}
\newcommand{\bs}{\begin{sub}}
\newcommand{\es}{\end{sub}}
\newcommand{\bsn}{\begin{subn}}
\newcommand{\esn}{\end{subn}}
\newcommand{\bea}{\begin{eqnarray}}
\newcommand{\eea}{\end{eqnarray}}
\newcommand{\bean}{\begin{eqnarray*}}
\newcommand{\eean}{\end{eqnarray*}}
\newcommand{\BA}[1]{\begin{array}{#1}}
\newcommand{\EA}{\end{array}}
\newlength{\wex}  \newlength{\hex}
\def\ga{\alpha}

            \def\gl{\lambda}

\def\squarebox#1{\hbox to #1{\hfill\vbox to #1{\vfill}}}

\newcommand{\Green}[4]{\mbox{$G^{#1}_{#2}(#3,#4)$}}

\title[Ratio limit theorems for heat kernels]
{On some Strong Ratio Limit Theorems for Heat Kernels}

\author[Martin Fraas David Krej\v{c}i\v{r}\'{\i}k and Yehuda Pinchover]{}

\keywords{Heat kernel, large time behavior, positive solutions, ratio limit theorem.
\medskip \\
This paper coincides with that published in
Discrete Contin.\ Dynam.\ Systems A \textbf{28} (2010), 495--509,
except for Remark 4 added after the paper has appeared.
}

\email{martin.fraas@gmail.com}
\email{krejcirik@ujf.cas.cz}
\email{pincho@techunix.technion.ac.il}

\dedicatory{Dedicated to Louis Nirenberg on the occasion of his 85th
 birthday}

\begin{document}
\maketitle

\centerline{\scshape Martin Fraas}
\medskip
{\footnotesize
\centerline{Department of Physics,}
\centerline{Technion - Israel Institute of
Technology,}
\centerline{Haifa, Israel}
}
% \email{martin.fraas@gmail.com}
\medskip

\centerline{\scshape David Krej\v{c}i\v{r}\'{\i}k}

\medskip

{\footnotesize
\centerline{Department of Theoretical Physics,}
\centerline{Nuclear Physics Institute,}
\centerline{Academy
of Sciences, 25068 \v{R}e\v{z}, Czech Republic}
}
%\email{krejcirik@ujf.cas.cz}
\medskip

\centerline{\scshape Yehuda Pinchover}

\medskip
{\footnotesize
\centerline{Department of Mathematics,}
\centerline{Technion - Israel Institute of
Technology,}
\centerline{Haifa, Israel}
%\email{pincho@techunix.technion.ac.il}
}
\bigskip

%\begin{document}

\begin{abstract}
We study strong ratio limit properties of the quotients of the heat kernels
of subcritical and critical operators which are defined on a noncompact Riemannian manifold.
\end{abstract}

\section{Introduction}\label{Introduction}\label{sect1}

Let $M$ be a connected noncompact Riemannian manifold, and let
$k_P^M(x,y,t)$ be the positive minimal (Dirichlet) heat kernel
associated with the parabolic equation
\begin{equation}\label{heat}
  u_t + Pu =0
  \qquad\mbox{on}\qquad M\times(0,\infty)
  \,,
\end{equation}
where~$P$ is a second-order elliptic differential operator on~$M$.
The coefficients of~$P$ are assumed to be real but~$P$ is not
necessarily symmetric. By definition, $(x,t) \mapsto k_P^M(x,y,t)$
is the minimal positive solution of~\eqref{heat}, subject to  the
initial data $\delta_y$, the Dirac distribution at $y \in M$. We
say that the operator~$P$ is \emph{subcritical} (respectively,
\emph{critical}) in $M$ if for some $x \not = y$
\begin{equation}\label{def.critical}
  \int_0^\infty k_P^{M}(x,y,\tau)\,\mathrm{d}\tau<\infty \qquad
  \left(\mbox{respectively, } \int_0^\infty
  k_P^{M}(x,y,\tau)\,\mathrm{d}\tau=\infty\right).
\end{equation}

In this paper we are concerned with the large time behavior of the
heat kernel~$k_P^M$ with regards to the criticality versus
subcriticality property of the operator~$P$. Since for any fixed
$x,y\in M$, $x \not = y$, we have that $k_{P}^M(x,y,\cdot)\in L^1(\mathbb{R}_+)$
if and only if $P$ is subcritical, it is natural to conjecture
that \emph{under some assumptions} the heat kernel of a subcritical operator~$P_+$~in $M$ decays
(in time) faster than the heat kernel of a critical
operator~$P_0$~in $M$. More precisely, we are interested to study
the following conjecture.
\begin{conjecture}\label{conjMain}
Let $P_+$ and $P_0$ be respectively subcritical and critical
operators in $M$. Then
%(under some additional conditions)
%
\begin{equation}\label{eqconjMain}
\lim_{t\to\infty}\frac{k_{P_+}^M(x,y,t)}{k_{P_0}^M(x,y,t)}=0
\end{equation}
locally uniformly in $M\times M$.
\end{conjecture}

The relevance of this conjecture becomes clearer
if we recall the relationship of~\eqref{def.critical}
to properties of positive solutions
of the elliptic equation
\begin{equation}\label{ell}
  Pu =0
  \qquad\mbox{on}\qquad M .
\end{equation}
Denote the cone of all positive (weak) solutions of~\eqref{ell}
by $\mathcal{C}_{P}(M)$.
The {\em generalized principal eigenvalue} of~$P$ in~$M$
is defined by
\begin{equation}\label{ev}
  \gl_0=\gl_0(P,M)
  := \sup\{\gl \in \mathbb{R} \mid
  \mathcal{C}_{P-\lambda}(M)\neq \emptyset\}
  \,.
\end{equation}
Throughout this paper we always assume that
$$
  \lambda_0\geq 0
$$
(actually, as it will become clear below, it is
enough to assume that $\lambda_0>-\infty$).

Recall that if $\lambda\!<\!\lambda_0$, then $P\!-\!\lambda$ is
subcritical in $M$, and for $\lambda\leq \lambda_0$, we
have $k_{P\!-\!\lambda}^M(x,y,t)\!=\!\mathrm{e}^{\lambda t}k_P^M(x,y,t)$.
It follows that $\gl_0(P_0,M)=0$ for any critical operator~$P_0$ in~$M$, while
$\gl_0(P_+,M) \geq 0$ for any subcritical operator~$P_+$ in~$M$.

It is well known that if $P$ is subcritical in $M$, then $P$ admits a {\em positive minimal Green function} $\Green{M}{P}{x}{y}$ which is given by
\begin{equation}\label{def.subcritical}
  \Green{M}{P}{x}{y}=\int_0^\infty k_P^{M}(x,y,\tau)\,\mathrm{d}\tau.
\end{equation}
On the other hand, if $P$ is critical in $M$, then $P$ does not admit a positive minimal Green function,
but admits a distinguished {\em unique} positive solution $\varphi
\in\mathcal{C}_{P}(M)$ satisfying $\varphi(x_0)=1$, where $x_0\in
M$ is a reference point. Such a solution is called a {\em ground
state} of the operator $P$ in $M$ \cite{Agmon82,Pheat,pinsky}. A ground state is characterized by being a global positive solution of the equation $Pu=0$ on $M$ of
{\em minimal growth in a neighborhood of infinity in $M$} \cite{Agmon82}. On the other hand, if $P$ is subcritical in $M$, then for any fixed $y\in M$, the positive minimal Green function $\Green{M}{P}{\cdot}{y}$ is a positive solution of the equation $Pu=0$ on $M\setminus\{y\}$ of
minimal growth in a neighborhood of infinity in $M$.
We also note that $P$ is critical in $M$ if and only the equation $Pu=0$ on $M$
admits (up to a multiplicative constant) a {\em unique positive supersolution}. Furthermore, $P$ is
critical (respectively, subcritical) in~$M$, if and only if $P^*$
(the formal adjoint of $P$) is critical (respectively,
subcritical) in $M$. The ground state of $P^*$ is denoted by
$\varphi^*$.

A critical operator~$P$ is said to be {\em positive-critical}
in~$M$ if $\varphi^*\varphi\in L^1(M)$, and {\em null-critical} in
$M$ if $\varphi^*\varphi\not\in L^1(M)$. The large time behavior
of the heat kernel of a {\em general} elliptic operator $P$ (with
$\lambda_0\geq 0$) is governed by the following theorem.
\begin{theorem}[{\cite{Pheat,P03}}]\label{mainthmhk}
Let $x,y\in M$. Then
 \begin{equation*}
\lim_{t\to\infty} \mathrm{e}^{\lambda_0
t}k_P^{M}(x,y,t)\!=\!
  \begin{cases}
    \dfrac{\varphi(x)\varphi^*(y)}{\int_M\!
\varphi(z)\varphi^*(z)\,\mathrm{d}\mu(z)} & \text{if } P\!-\!\lambda_0
\text{ is positive-critical},
\\[3mm]
    0 & \text{otherwise}.
  \end{cases}
 \end{equation*}
Furthermore,
\begin{equation}\label{eqgreen}
\lim_{t\to\infty} \mathrm{e}^{\lambda_0
t}k_P^{M}(x,y,t)=
\lim_{\lambda\nearrow\lambda_0}(\lambda_0-\lambda)\Green{M}{P-\lambda}{x}{y}.
\end{equation}
\end{theorem}
As a consequence of this theorem, we see that $\lim_{t\to\infty}
\mathrm{e}^{\lambda_0 t}k_P^{M}(x,y,t)$ always exists. On the other hand, heat kernels might have slow decay (see for example \cite{BCG} and the references therein). Therefore, it is natural to ask how fast versus slow this limit is approached, and in particular,
to examine the validity of Conjecture~\ref{conjMain}. We note that
Theorem~\ref{mainthmhk} implies that Conjecture~\ref{conjMain}
obviously holds true if~$P_0$ is positive-critical.

In \cite[Theorems~4.2 and 4.4]{M_large_time} M.~Murata
obtained the exact asymptotic for the heat kernels of
nonnegative Schr\"odinger operators with {\em short-range} (real)
potentials defined on~$\mathbb{R}^d$, $d\geq 1$.
These results imply that Conjecture~\ref{conjMain}
holds true for such operators.

The aim of the present paper is to discuss Conjecture~\ref{conjMain}
and closely related problems in the {\em general} case,
and to obtain some results under minimal assumptions.

Our study is motivated by a recent paper
\cite{KZ} by D.~Krej\v{c}i\v{r}\'{\i}k and E.~Zuazua,
where it is conjectured that for selfadjoint subcritical and critical operators $P_+$ and $P_0$, respectively, defined on $L^2(M,\mathrm{d}x)$ one has
\begin{equation}\label{KZqconj}
\lim_{t\to\infty}\frac{\|\mathrm{e}^{-P_+t}\|_{L^2(M,W\,\mathrm{d}x)
\to L^2(M,\mathrm{d}x)}}
{\|\mathrm{e}^{-P_0t}\|_{L^2(M,W\,\mathrm{d}x)\to L^2(M,\mathrm{d}x)}} =0
\end{equation}
for some positive weight function $W$. In fact, the above conjecture is proved in
\cite{KZ} for the Dirichlet Laplacian defined on a special class of quasi-cylindrical domains.

It turns out that Conjecture~\ref{conjMain} is related to the
following conjecture raised by E.~B.~Davies \cite{D} in the
self-adjoint case.
\begin{conjecture}[Davies' Conjecture]\label{conjD}
Let $Lu=u_t+P(x, \partial_x)u$ be a parabolic operator which is
defined on a noncompact Riemannian manifold $M$. Fix reference
points $x_0, y_0\in M$. Then
\begin{equation}\label{eqconjD}
\lim_{t\to\infty}\frac{k_P^M(x,y,t)}{k_P^M(x_0,y_0,t)}=a(x,y)
\end{equation}
exists and is positive for all $x,y\in M$, Moreover, for any fixed $y\in M$ we have $a(\cdot,y)\in \mathcal{C}_{P-\gl_0}(M)$.
Similarly, for a fixed  $x\in M$ we have  $a(x,\cdot)\in \mathcal{C}_{P^*-\gl_0}(M)$ (see also \cite{PDavies} and the references therein).
\end{conjecture}
\begin{remark}\label{dim1}
Obviously, Davies' Conjecture holds if $P$ is positive-critical.
Moreover, it holds true in the symmetric case
(for a precise definition of $P$ being symmetric
see Section~\ref{sec_preliminar})
if $\dim \mathcal{C}_{P}(M) = 1$ \cite[Corollary 2.7]{ABJ}.
In particular, it holds true for a critical symmetric operator.
For a probabilistic interpretation of Conjecture~\ref{conjD}, see \cite{ABJ}.

On the other hand, G.~Kozma announced \cite{Kozma}
that he constructed a graph $G$ such that for some two vertices $x,y \in G$
the sequence $\{k(x,x,n)/k(y,y,n)\}_{n=1}^\infty$
of the ratio of the corresponding heat kernel does not converge as $n\to \infty$.
\end{remark}

The organization of this paper is as follows. In the following
section, we give a precise definition of the operator~$P$ in~$M$
and introduce the necessary background to study
Conjecture~\ref{conjMain}.
In Section~\ref{secmain}, we
prove (under some additional assumptions) Conjecture~\ref{conjMain} in the symmetric case
(Theorem~\ref{mainthmFKP}).
In particular, Theorem~\ref{mainthmFKP} provides an affirmative answer to
the conjecture in the case of positive
perturbations (Corollary~\ref{nonnegpert}).
The relationship between Davies' conjecture and Conjecture~\ref{conjMain}
is examined for nonsymmetric operators in Section~\ref{secDavies}.
Two regimes are considered:
positive perturbations (Theorem~\ref{thm_nonselfadj})
and semismall perturbations (Theorem~\ref{thm_ssp}).
We conclude the paper in Section~\ref{secequiv}, where we ask a
general question concerning the equivalence of heat kernels on Riemannian manifolds and
provide sufficient conditions for the validity of a principal
hypothesis of theorems~\ref{thmlpluspos}, \ref{mainthmFKP}, and~\ref{thm_ssp}.

%%%%%%%%%%%%%%%%%%%%%%%%%%%%%%%%%%%%%%%%%%
\section{Preliminaries}\label{sec_preliminar}
%%%%%%%%%%%%%%%%%%%%%%%%%%%%%%%%%%%%%%%%%%
%
Let $M$ be a smooth connected noncompact Riemannian manifold of
dimension~$d$. We recall the definition of a weighted manifold
associated with~$M$. Denote by $\mathrm{d}x$ the Riemannian
density on~$M$. The divergence and gradient with respect to the
Riemannian  metric on~$M$ are denoted by $\mathrm{div}$ and
$\nabla$, respectively. Let~$m$ be a positive measurable function
on~$M$ such that~$m$ and~$m^{-1}$ are bounded on any compact
subset of~$M$. Set $\mathrm{d}\mu:=m \mathrm{d}x$. The couple
$(M,\mathrm{d}\mu)$ is called a \emph{weighted manifold} over which we
consider the Lebesgue spaces $L^p(M,\mathrm{d}\mu)$.

We associate to $M$ {\em an exhaustion}, i.e. a sequence of smooth,
relatively compact domains $\{M_{j}\}_{j=1}^{\infty}$ such that
$M_1\neq \emptyset$, $\overline{M}_{j}\subset M_{j+1}$ and
$\cup_{j=1}^{\infty}M_{j}=\Omega$. For every $j\geq 1$, we denote
$M_{j}^*:=\Omega\setminus \overline{M_j}$.

We consider a second-order elliptic differential
operator~$P$ which is defined on $(M,d\mu)$ by
\begin{equation}\label{P}
Pu:=-m^{-1} \mathrm{div}(mA\nabla u -muC) -\langle B,\nabla u\rangle + Du,
\end{equation}
where  $D$ is a real-valued measurable function on $M$,
$B$ and $C$ are measurable vector fields on~$M$,
and $A$ is a symmetric locally bounded measurable section on~$M$
of $\mathrm{End}(TM)$ such that $P$ is locally uniformly elliptic on $M$.
Here $T_xM$, $TM$, $\mathrm{End}(T_xM)$ and $\mathrm{End}(TM)$ denote the tangent space
to $M$ at $x\in M$, the tangent bundle, the endomorphisms on $T_xM$
and the corresponding bundle, respectively.
The inner product and the induced norm~on $TM$
is denoted by $\langle \cdot,\cdot \rangle$ and $|\cdot|$, respectively.
We assume that   $D, |B|^2,|C|^2 \in L^p_{\mathrm{loc}}(M,\mathrm{d}\mu)$
for some $p >\max\{n/2,1\}$.

We say that~$P$ is \emph{symmetric} if $B=C=0$ on $M$. So, in the symmetric case $P$ has the form
\begin{equation}\label{PS}
Pu=-m^{-1} \mathrm{div}(mA\nabla u) +Du.
\end{equation}

 The reason for
this terminology is that the minimal operator constructed from the formal differential operator~\eqref{PS}, \emph{i.e.}\
the restriction of~$P$ to $C_0^\infty(M)$, is symmetric in
$L^2(M,\mathrm{d}\mu)$. The Friedrichs extension of the minimal
operator defines a self-adjoint operator in
$L^2(M,\mathrm{d}\mu)$; it acts weakly as~\eqref{PS} and satisfies
Dirichlet boundary conditions on~$\partial M$ in a generalized
sense. By definition, it is the operator associated with the
closure of the quadratic form~$Q$ in $L^2(M,\mathrm{d}\mu)$
defined by
\begin{equation}\label{Q}
Q[u]:= \int_M \left(\langle A\nabla u, \nabla u \rangle  +
D|u|^2\right) \mathrm{d}\mu  \qquad u \in  C_0^\infty(M)
\,.
\end{equation}
It is well known that for such operators we have
$$\gl_0= \inf\left\{
Q[u] \, \Big| \ u\in C_0^\infty(M),\; \int_M |u|^2\,
\mathrm{d}\mu = 1 \right\} \,,
$$
where~$\lambda_0$ is the generalized principal eigenvalue of~$P$
introduced in~\eqref{ev}. In other words, $\gl_0$~equals to the
bottom of the spectrum of the Friedrichs extension if $P$~is symmetric.
\begin{remark}\label{6Rem21}
Let $t_n\to \infty$. By a standard parabolic argument, we may
extract a subsequence $\{t_{n_k}\}$ such that for every $x,y\in
M$ and $s<0$
\begin{equation}\label{eqsequence}
a(x,y,s):=\lim_{k\to\infty}\frac{k_P^M(x,y,s+t_{n_k})}{k_P^M(x_0,y_0,t_{n_k})}
 \end{equation}
 exists. Moreover, $a(\cdot,y,\cdot)\in
\mathcal{H}_P(M\times \mathbb{R}_-)$, where $\mathcal{H}_P(M\times (a,b))$ denotes the cone of all
nonnegative solutions of the equation~\eqref{heat} in $M\times (a,b)$.
Note that in the
selfadjoint case, the above is valid for all $s\in \mathbb{R}$ \cite{PDavies}.
\end{remark}

Now we recall some auxiliary results which we will need in the sequel.
First, we mention convexity properties of heat kernels.
\begin{lemma}\label{lem_convex}
Consider the following one-parameter family of elliptic operators
$$P_\ga:= P+ \ga V\qquad 0\leq \ga\leq 1,$$
where $V$ is a nonzero potential. Assume that $\gl_0(P_\ga,M)\geq 0$ for $\ga=0,1$. Then $\gl_0(P_\ga,M)\geq 0$ for $0\leq \ga\leq 1$, and the corresponding heat kernels satisfy the inequality
\begin{equation}\label{eqkconvex}
k_{P_\ga}^M(x,y,t)\leq [k_{P_0}^M(x,y,t)]^{1-\ga}[k_{P_1}^M(x,y,t)]^\ga \quad \forall\, x,y\in M,\; t>0,\;  0\leq \ga\leq 1.
\end{equation}
Moreover, $P_\ga$ is subcritical for any $0<\ga<1$.
\end{lemma}
For a proof of the lemma see \cite{P90}. In particular, \eqref{eqkconvex} is proved by applying H\"older's inequality to the Feynmann-Kac formula (see e.g., \cite[Lemma~B.7.7]{Simon82}).

 We also need the following key lemma
\begin{lemma}\label{corskeleton}
Assume that $\lambda_0(P,M)\geq 0$, and that either $P$ is symmetric or that Davies' conjecture holds for $P$ in $M$.   Then for any fixed  $x, y \in M$ we have
\begin{equation}\label{eqconjDw9}
\lim_{t\to\infty}\frac{k_P^M(x,y,\tau+t)}{k_P^M(x,y,t)}=\mathrm{e}^{-\lambda_0 \tau}
\qquad \forall  \tau\in \mathbb{R}_-.
\end{equation}
\end{lemma}
\begin{proof}
If $P$ is symmetric, then the function $t\mapsto
k_P^M(x,x,t)$ is log-convex, and therefore the lemma follows by a
polarization argument (see for example \cite{CMS,D}).

Suppose now that Davies' conjecture holds for $P$ in $M$. Then as in the proof of \cite[Theorem 3.1]{PDavies}, fix $y\in M$ and let $\{t_n\}$ be a sequence such that $t_n\to \infty$. Consider the sequence
$\{\frac{k_P^M(x,y,\tau+t_n)}{k_P^M(y,y,t_n)}\}$
that converges (up to a subsequence) to a nonnegative solution $K_P^M(x,\tau)\in \mathcal{H}_P(M\times \mathbb{R}_-)$ (see Remark~\ref{6Rem21}). By
our assumption, for any $\tau$ we have
$$\lim_{n\to\infty}\frac{k_P^M(x,y,\tau+t_n)}{k_P^M(y,y,\tau+t_n)}=
\lim_{s\to\infty}\frac{k_P^M(x,y,s)}{k_P^M(y,y,s)}=\frac{a(x,y)}{a(y,y)}=:
b(x)>0,$$ where $b\in \mathcal{C}_{P-\lambda_0}(M)$, and $b$ does not depend on the sequence $\{t_n\}$.

On the other hand,
$$\lim_{n\to\infty}\frac{k_P^M(y,y,\tau+t_n)}{k_P^M(y,y,t_n)}=
K_P^M(y,\tau)=:f(\tau).$$
 Since
$$\frac{k_P^M(x,y,\tau+t_n)}{k_P^M(y,y,t_n)}=
\frac{k_P^M(x,y,\tau+t_n)}{k_P^M(y,y,\tau+t_n)}\cdot
\frac{k_P^M(y,y,\tau+t_n)}{k_P^M(y,y,t_n)},$$ we
have
%%%%%%%%%%%%%%
$$K_P^M(x,\tau)=b(x)f(\tau).$$
 Since $K_P^M(x,\tau)$ solves the parabolic equation $u_\tau + Pu =0$ in $M\times \mathbb{R}_-$, and $b\in \mathcal{C}_{P-\lambda_0}(M)$, it follows that $f$ solves the initial value problem (backwards in time)
$$\qquad f'+\lambda_0 f=0 \quad \mbox{ on } \;\; \mathbb{R}_-,\;\; f(0)=1.$$
 In particular, $f$ does not depend on the sequence
$\{t_n\}$. Thus, $$ \lim_{t\to \infty}
\frac{k_P^M(y,y,\tau+t)}{k_P^M(y,y,t)}=f(\tau)=\mathrm{e}^{-\lambda_0
\tau}.$$
Finally,
\begin{equation*}
 \begin{split}
\lim_{t\to \infty}\frac{k_P^M(x,y,\tau+t)}{k_P^M(x,y,t)}=\lim_{t\to \infty}
\frac{k_P^M(y,y,\tau+t)}{k_P^M(y,y,t)}\cdot
\frac{k_P^M(x,y,\tau+t)}{k_P^M(y,y,\tau+t)}\cdot
\frac{k_P^M(y,y,t)}{k_P^M(x,y,t)}\\[3mm]
=\mathrm{e}^{-\lambda_0
\tau}\cdot b(x)\cdot (b(x))^{-1}=\mathrm{e}^{-\lambda_0
\tau}.
 \end{split}
\end{equation*}
\end{proof}
%%%%%%%%%%%%%%%%%%%%
 It turns out that Lemma~\ref{corskeleton} implies that the case $\lambda_0(P_+,M)> 0$ is easier than the case  $\lambda_0(P_+,M)=0$.  Moreover, if  $\lambda_0(P_+,M)> 0$, then the  assumptions that we need for the validity of Conjecture~\ref{conjMain} are weaker. We have
%%%%%%%%%%%%%%%%%%%%%%%%%%%
\begin{theorem}\label{thmlpluspos}
Let $P_0$ be critical operator in $M$, and let $P_+$ be a subcritical operator in $M$ satisfying $\lambda_+:=\lambda_0(P_+,M)> 0$. Suppose that either $P_0$ and $P_+$ are symmetric operators, or that
Davies' conjecture (Conjecture~\ref{conjD}) holds true for both $k_{P_0}^M$ and $k_{P_+}^M$.

 Assume further that for some fixed $y_1\in M$ there exists a positive constant $C$ satisfying the following condition: for each $x\in M$ there
exists $T(x)>0$ such that
\begin{equation}\label{Ass1m}
    k_{P_+}^M(x,y_1,t)\leq C k_{P_0}^M(x,y_1,t)\qquad \forall t>T(x).
\end{equation}
Then
\begin{equation}\label{eqconjMain1m}
\lim_{t\to\infty}\frac{k_{P_+}^M(x,y,t)}{k_{P_0}^M(x,y,t)}=0
\end{equation}
locally uniformly in $M\times M$.
 \end{theorem}
%%%%%%%%%%%%%%%%%%%%%%%%%%
\begin{proof}  Fix $x\in M$, and $s\in \mathbb{R}_-$, and let $y_1\in M$ be the point satisfying  \eqref{Ass1m}. We have
\begin{equation}\label{eq:nm}
\frac{k_{P_+}^M(x,y_1,t)}{k_{P_0}^M(x,y_1,t)}=\frac{k_{P_+}^M(x,y_1,t+s)}{k_{P_0}^M(x,y_1,t+s)}\times \frac{k_{P_+}^M(x,y_1,t)}{k_{P_+}^M(x,y_1,t+s)}\times \frac{k_{P_0}^M(x,y_1,t+s)}{k_{P_0}^M(x,y_1,t)}.
\end{equation}
Recall that $\lambda_0(P_0,M)=0$, and by our assumption $\lambda_+ > 0$. By Lemma~\ref{corskeleton} we have  \begin{equation}\label{eq:n1m}
\lim_{t\to\infty } \frac{k_{P_+}^M(x,y_1,t)}{k_{P_+}^M(x,y_1,t+s)}= e^{\lambda_+ s},\qquad  \lim_{t\to\infty } \frac{k_{P_0}^M(x,y_1,t+s)}{k_{P_0}^M(x,y_1,t)}=1.
\end{equation}
Therefore, using \eqref{eq:n1m} and our assumption \eqref{Ass1m}, it follows from \eqref{eq:nm} that for $t$ sufficiently large we have
\begin{equation}\label{eq:n2m}
\frac{k_{P_+}^M(x,y_1,t)}{k_{P_0}^M(x,y_1,t)}\leq 2C e^{\lambda_+ s}.
\end{equation}
Since $s$ is an arbitrary negative number, \eqref{eq:n2m} implies that
\begin{equation}\label{eq:n3m}
\lim_{t\to\infty }\frac{k_{P_+}^M(x,y_1,t)}{k_{P_0}^M(x,y_1,t)}=0.
\end{equation}
The parabolic Harnack inequality and a standard parabolic regularity argument imply now that   \begin{equation*}
\lim_{t\to\infty}\frac{k_{P_+}^M(x,y,t)}{k_{P_0}^M(x,y,t)}=0
\end{equation*}
locally uniformly in $M\times M$.
\end{proof}
By the generalized maximum principle, assumption \eqref{Ass1m} is satisfied with $C=1$ if $P_+=P_0+V$ and $V$ is any {\em nonnegative} potential. In Section~\ref{secequiv}, we discuss some other conditions under which assumption \eqref{Ass1m} is satisfied.

We shall need also the following Liouville comparison theorem (see \cite{Pliouv}).
%%%%%%%%%%%%%%%%%%
\begin{theorem}\label{mainthmLt}
Let $P_0$ and $P_1$ be two symmetric operators defined on $M$ of the
form~\eqref{PS}. Assume that the following assumptions hold true.
\begin{itemize}
\item[(i)] The operator  $P_0$ is critical in $M$. Denote
by $\varphi\in \mathcal{C}_{P_0}(M)$ its ground state.

\item[(ii)]  $\lambda_0(P_1,M)\geq 0$, and there exists a
real function $\psi\in H^1_{\mathrm{loc}}(M)$ such that
$\psi_+\neq 0$, and $P_1\psi \leq 0$ in $M$, where
$u_+(x):=\max\{0, u(x)\}$.

\item[(iii)] Denote by $A_1, A_0$ the sections on $M$ of $\mathrm{End}(TM)$, and by
$m_1, m_0$ the weights corresponding to  $P_1, P_0$, respectively. The following matrix inequality holds
\begin{equation}\label{psialephia}
(\psi_+)^2(x) m_1(x)A_1(x)\leq C\varphi^2(x) m_0(x)A_0(x)\qquad  \mbox{for a.e. }
x\in  M,
\end{equation}
where $C>0$ is a positive constant.
\end{itemize}
Then the operator $P_1$ is critical in $M$, and $\psi$ is its
ground state. In particular, $\dim \mathcal{C}_{P_1}(M)=1$
and $\lambda_0(P_1,M)=0$.
\end{theorem}
%

%
%\begin{comment}
%\begin{rem}\label{reSPHI}{\em
%Heat kernels of symmetric operators satisfy the following strong parabolic Harnack inequality \cite{D}.
%For any compact set $K\Subset M$ and $T>0$ there exists a positive constant $C=C(K,T)$ such that
%\begin{equation}\label{SPHI}
%k_P^M(x_1,x_2,t)\leq C k_P^M(x_3,x_4,t)\qquad \forall x_1,x_2,x_3,x_4\in K, t\geq T.
%\end{equation}
% }
%\end{rem}
%\end{comment}
%
Let $f,g \in C(\Omega)$ be nonnegative functions, we use the notation $f\asymp g$ on
$\Omega$ if there exists a positive constant $C$ such that
$$C^{-1}g(x)\leq f(x) \leq Cg(x) \qquad \mbox{ for all } x\in \Omega.$$
In the sequel we shall need also to use results concerning small
and semismall perturbations.
%%%%%%%%%%%%%%%%%%%%%
These notions were introduced in \cite{P89} and \cite{Msemismall}
respectively, and are closely related to the stability of
$\mathcal{C}_P(\Omega)$ under perturbation by a potential $V$.
\begin{definition} \label{spertdef}
Let $P$ be a subcritical operator in $M$, and let $V$ be a
potential defined on $M$.

{\em (i)} We say that $V$ is a {\em small perturbation}
 of $P$ in $M$ if
\be \label{sperteq} \lim_{j\rightarrow \infty}\left\{\sup_{x,y\in
M_{j}^*} \int_{M_{j}^*}\frac{\Green{M}{P}{x}{z}|V(z)|
\Green{M}{P}{z}{y}}{\Green{M}{P}{x}{y}}\,\mathrm{d}\mu(z)\right\}=0.
\end{equation}

{\em (ii)} $V$ is a {\em  semismall perturbation} of
$P$ in $M$ if for some $x_0\in M$ we have
 \be \label{semisperteq} \lim_{j\rightarrow
\infty}\left\{\sup_{y\in M_{j}^*} \int_{M_{j}^*}
\frac{\Green{M}{P}{x_0}{z}|V(z)|\Green{M}{P}{z}{y}}
{\Green{M}{P}{x_0}{y}}\,\mathrm{d}\mu(z)\right\}=0. \end{equation}
 \end{definition}
 Recall that small perturbations are semismall \cite{Msemismall}. For semismall perturbations we have
\begin{theorem}[\cite{Msemismall,P89,P90}]\label{thmssp}
Let~$P$ be a subcritical operator in~$M$.
Assume that $V=V_+-V_-$
is a semismall perturbation of~$P^*$ in~$M$
satisfying~$V_- \not=0$, where  $V_\pm(x)=\max\{0, \pm V(x)\}$.

 Then there exists $\alpha_0>0$ such that $P_\alpha:= P+\alpha V$
 is subcritical  in $M$ for all $0\leq \alpha< \alpha_0$
 and critical for $\alpha=\alpha_0$.

 Moreover, let $\varphi$ be the
ground state of $P+\alpha_0 V$ and let $y_0$ be a fixed reference point in $M_1$.
Then  for any $0\leq \alpha< \alpha_0$
$$   \varphi   \asymp   \Green{M}{P_\alpha}{\cdot}{y_0} \qquad \mbox {in } M_1^*.
$$
\end{theorem}
%
%%%%%%%%%%%%%%%%%%%%%%%%%%%%%%%%%%%%%%%%%%%%%%
\section{The symmetric case}\label{secmain}
%%%%%%%%%%%%%%%%%%%%%%%%%%%%%%%%%%%%%%%%%%%%%
In this section we prove the following theorem.
\begin{theorem}\label{mainthmFKP}
Let the subcritical operator~$P_+$
and the critical operator~$P_0$ be symmetric in $M$.
Assume that $A_+$ and $A_0$, the sections on $M$ of $\mathrm{End}(TM)$, and the weights $m_+$ and $m_0$, corresponding to  $P_+$ and $P_0$, respectively, satisfy the following matrix inequality
\begin{equation}\label{A+leqA0}
m_+(x)A_+(x)\leq C m_0(x)A_0(x) \qquad  \mbox{for a.e. } x\in M,
\end{equation}
where $C$ is a positive constant. Assume further that condition \eqref{Ass1m} holds true.
Then
\begin{equation}\label{eqconjMain1}
\lim_{t\to\infty}\frac{k_{P_+}^M(x,y,t)}{k_{P_0}^M(x,y,t)}=0
\end{equation}
locally uniformly in $M\times M$.
 \end{theorem}
\begin{proof}
 By Theorem~\ref{thmlpluspos}, we may assume that $\lambda_0(P_+,M)= 0$.

Assume to the contrary that for some $x_0,y_0\in M$ there exists a sequence $\{t_n\}$ such that $t_n \to \infty$ and
    \begin{equation}\label{eqconjMain2}
\lim_{n\to\infty}\frac{k_{P_+}^M(x_0,y_0,t_n)}{k_{P_0}^M(x_0,y_0,t_n)}=K>0.
 \end{equation}
Consider the sequence of functions $\{u_n\}_{n=1}^\infty$ defined by
$$ u_n(x,s):=\frac{k_{P_+}^M(x,y_0,t_n+s)}{k_{P_0}^M(x_0,y_0,t_n)}\qquad x\in M,\; s\in \mathbb{R}.$$
We note that
$$ u_n(x,s)=\frac{k_{P_+}^M(x,y_0,t_n+s)}{k_{P_+}^M(x_0,y_0,t_n)}\times \frac{k_{P_+}^M(x_0,y_0,t_n)}{k_{P_0}^M(x_0,y_0,t_n)}\;.$$
Therefore, by assumption \eqref{eqconjMain2} and Remark~\ref{6Rem21} it follows that we may subtract a subsequence which we rename by $\{u_n\}$ such that
$$\lim_{n\to\infty} u_n(x,s)=u_+(x,s),$$
where $u_+\in \mathcal{H}_{P_+}(M\times \mathbb{R})$ and $u_+\gneqq 0$.

On the other hand,
\begin{align*}
v_n(x):=\frac{k_{P_+}^M(x,y_0,t_n)}{k_{P_0}^M(x_0,y_0,t_n)}
=u_n(x,s)\frac{k_{P_+}^M(x,y_0,t_n)}{k_{P_+}^M(x,y_0,t_n+s)}\;.
\end{align*}
By our assumption, $\lambda_0(P_+,M)= 0$, therefore Lemma~\ref{corskeleton} implies that
$$\lim_{n\to\infty}\frac{k_{P_+}^M(x,y_0,t_n)}{k_{P_+}^M(x,y_0,t_n+s)}=1.$$
Therefore,
$$\lim_{n\to\infty}v_n(x)=\lim_{n\to\infty}u_n(x,s)=u_+(x,s),$$
and $u_+$ does not depend on $s$, and hence $u_+$ is a positive solution of the elliptic
equation $P_+u=0$ in $M$ and we have
\begin{equation}\label{eq0}
    \lim_{n\to\infty}\frac{k_{P_+}^M(x,y_0,t_n)}{k_{P_0}^M(x_0,y_0,t_n)}=u_+(x).
\end{equation}
On the other hand, by Remark~\ref{dim1} we have
\begin{equation}\label{eq1}
\lim_{n\to\infty}\frac{k_{P_0}^M(x,y_0,t_n)}{k_{P_0}^M(x_0,y_0,t_n)}
= \frac{\varphi(x)}{\varphi(x_0)} =: u_0(x),
\end{equation}
where $\varphi$ is the ground state of $P_0$.

Combining \eqref{eq0} and \eqref{eq1}, we obtain
 \begin{align}\label{eq2}
\lim_{n\to\infty}  \frac{k_{P_+}^M(x,y_0,t_n)}{k_{P_0}^M(x,y_0,t_n)}= \lim_{n\to\infty} \left\{
 \dfrac{\frac{k_{P_+}^M(x,y_0,t_n)}{k_{P_0}^M(x_0,y_0,t_n)}}
 {\frac{k_{P_0}^M(x,y_0,t_n)}{k_{P_0}^M(x_0,y_0,t_n)}}\right\}=\frac{u_+(x)}{u_0(x)}.
 \end{align}
 On the other hand, by assumption \eqref{Ass1m} and the parabolic Harnack inequality
there exists a positive constant~$C_1$
which depends on $P_+, P_0, y_0, y_1$
such that
\begin{multline}\label{Ass2}
    C_1^{-1}k_{P_+}^M(x,y_0,t-1)\leq k_{P_+}^M(x,y_1,t)\\\leq C k_{P_0}^M(x,y_1,t) \leq
CC_1 k_{P_0}^M(x,y_0,t+1)    \quad \forall x\in M, t>T(x).
\end{multline}
Moreover, by Lemma~\ref {corskeleton} we have
\begin{equation}\label{eqconjDw91}
\lim_{t\to\infty}\frac{k_{P_+}^M(x,y_0,t-1)}{k_{P_+}^M(x,y_0,t)}=1, \quad \mbox{and}\quad \lim_{t\to\infty}\frac{k_{P_0}^M(x,y_0,t+1)}{k_{P_0}^M(x,y_0,t)}=1
\qquad \forall  x\in M.
\end{equation}
Therefore, \eqref{Ass2} and \eqref{eqconjDw91} imply that there exists $C_0>0$ such that
\begin{equation}\label{Ass3}
    k_{P_+}^M(x,y_0,t)\leq C_0 k_{P_0}^M(x,y_0,t)    \qquad \forall x\in M, t>T(x).
\end{equation}
 Consequently, \eqref{eq2} and \eqref{Ass3} imply that
 $$u_+(x)\leq C_0 u_0(x) = \tilde{C}_0 \varphi(x)\qquad \forall x\in M. $$
 Therefore, using \eqref{A+leqA0} we obtain
 \begin{equation}\label{u+leu_0}
(u_+)^2(x) m_+(x)A_+(x)\leq C_2\varphi^2(x) m_0(x)A_0(x)\qquad  \mbox{for a.e. } x\in M,
\end{equation}
where $C_2>0$ is a positive constant. Thus, Theorem~\ref{mainthmLt} implies that $P_+$ is critical in $M$ which is a contradiction.
The last statement of the theorem follows from the parabolic
Harnack inequality and parabolic regularity.
  \end{proof}
By the generalized maximum principle, assumption \eqref{Ass1m} in Theorem~\ref{mainthmFKP} is satisfied with $C=1$ if $P_+=P_0+V$, where  $P_0$ is a critical operator on $M$ and $V$ is any {\em nonnegative} potential. Note that if the potential is in addition nontrivial,
then $P_+$ is indeed subcritical in $M$. Therefore, we have
\begin{corollary}\label{nonnegpert}
Let $P_0$ be a symmetric operator which is critical in $M$,
and let $P_+:=P_0+V$, where  $V$ is a nonzero nonnegative potential.  Then
\begin{equation}\label{eqconjMain5}
\lim_{t\to\infty}\frac{k_{P_+}^M(x,y,t)}{k_{P_0}^M(x,y,t)}=0
\end{equation}
locally uniformly in $M\times M$.
\end{corollary}
\begin{remark}
The pointwise limit~\eqref{eqconjMain1} of Theorem~\ref{mainthmFKP}
leads to a normwise limit of the type~\eqref{KZqconj}
in suitably chosen functional spaces.
Let us assume that the initial data~$u_0$ of~\eqref{heat}
lie in the space $L_0^1(M)$ of compactly supported
integrable functions on~$M$ equipped with the usual $L^1$-norm.
Since $\mathrm{e}^{-P_+ t}$ and $\mathrm{e}^{-P_0 t}$
are positivity-preserving
under the hypotheses of Theorem~\ref{mainthmFKP},
we can restrict ourselves to $u_0 \geq 0$.
For any $x \in M$, we have
$$
  \mathrm{e}^{-P_+t} u_0(x) =\!\!
  \int_M \!\!k_+(x,y,t) \, u_0(y) \, \mathrm{d}\mu(y)
  \leq \!
  \left\{
  \sup_{y\in\mathrm{supp}(u_0)} \frac{k_+(x,y,t)}{k_0(x,y,t)}
  \right\}
  \mathrm{e}^{-P_0 t} u_0(x)
  \,.
$$
Consequently, for any compact set $K \Subset M$, we arrive at
$$
  \frac{\|\mathrm{e}^{-P_+t}\|_{L_0^1(M) \to L^\infty(K)}}
  {\|\mathrm{e}^{-P_0 t}\|_{L_0^1(M) \to L^\infty(K)}}
  \leq
  \sup_{x \in K, \ y\in\mathrm{supp}(u_0)} \frac{k_+(x,y,t)}{k_0(x,y,t)}
  \xrightarrow[t \to 0]{} 0
$$
by Theorem~\ref{mainthmFKP}.
\end{remark}
%

%%%%%%%%%%%%%%
\section{Davies' conjecture and Conjecture~\ref{conjMain}}\label{secDavies}
%%%%%%%%%%%
In the present section we discuss the nonsymmetric case. We study two cases where Davies' conjecture imply Conjecture~\ref{conjMain}. First, we show that in the nonsymmetric case, the result of Corollary~\ref{nonnegpert} for positive perturbations of a critical operator~$P_0$ still holds provided that the validity of Davies' conjecture (Conjecture~\ref{conjD})
is assumed instead of the symmetry hypothesis.
More precisely, we have
\begin{theorem}\label{thm_nonselfadj}
Let $P_0$ be a critical operator in $M$, and let $P_+=P_0+V$, where~$V$ is any nonzero nonnegative potential on $M$. Assume that
Davies' conjecture (Conjecture~\ref{conjD}) holds true for both $k_{P_0}^M$ and $k_{P_+}^M$. Then
\begin{equation}\label{eqconjMain6}
\lim_{t\to\infty}\frac{k_{P_+}^M(x,y,t)}{k_{P_0}^M(x,y,t)}=0
\end{equation}
locally uniformly in $M\times M$.
 \end{theorem}
\begin{proof}
 By Theorem~\ref{thmlpluspos}, we may assume that $\lambda_0(P_+,M)= 0$.

Assume to the contrary that for some $x_0,y_0\in M$ there exists a sequence $\{t_n\}$ such that $t_n \to \infty$ and
    \begin{equation}\label{eqconjMain2n}
\lim_{n\to\infty}\frac{k_{P_+}^M(x_0,y_0,t_n)}{k_{P_0}^M(x_0,y_0,t_n)}=K>0.
 \end{equation}
Consider the functions $v_+$ and  $v_0$ defined by
$$ v_+(x,t):=\frac{k_{P_+}^M(x,y_0,t)}{k_{P_+}^M(x_0,y_0,t)}\,,\quad v_0(x,t):=\frac{k_{P_0}^M(x,y_0,t)}{k_{P_0}^M(x_0,y_0,t)} \qquad x\in M, t>0.$$
By our assumption,
$$\lim_{t\to\infty} v_+(x,t)=u_+(x),\qquad \lim_{t\to\infty}
v_0(x,t)=u_0(x),$$
where $u_+\in \mathcal{C}_{P_+}(M)$ and $u_0\in \mathcal{C}_{P_0}(M)$.

On the other hand, by the generalized maximum principle
\begin{equation}\label{gmp}
\frac{k_{P_+}^M(x,y_0,t)}{k_{P_0}^M(x,y_0,t)}\leq 1.
\end{equation}
Therefore,
\begin{equation}\label{eqm1}
\frac{k_{P_+}^M(x_0,y_0,t_n)}{k_{P_0}^M(x_0,y_0,t_n)}\times
\dfrac{\frac{k_{P_+}^M(x,y_0,t_n)}{k_{P_+}^M(x_0,y_0,t_n)}}
{\frac{k_{P_0}^M(x,y_0,t_n)}{k_{P_0}^M(x_0,y_0,t_n)}} =
  \frac{k_{P_+}^M(x,y_0,t_n)}{k_{P_0}^M(x,y_0,t_n)}\leq 1.
\end{equation}
Letting $n\to \infty$ we obtain
\begin{equation}\label{eqm2}
Ku_+(x) \leq u_0(x) \qquad x\in M.
\end{equation}
It follows that $v(x):=u_0(x)-Ku_+(x)$ is a nonnegative supersolution of the  equation $P_0 u=0$ in $M$ which is not a solution. In particular, $v\neq 0$. By the strong maximum principle $v(x)$ is a strictly positive supersolution of the  equation $P_0 u=0$ in $M$ which is not a solution. This contradicts the criticality of $P_0$ in $M$.
\end{proof}
%%%%%%%%%%%%%%%%%%%%%%%%%%%%%%%%%%%%%%%%%%%%%%
The second result concerns semismall perturbations.
\begin{theorem}\label{thm_ssp}
Let $P_+$ and $P_0=P_++V$ be a subcritical operator and a critical operator in $M$, respectively. Suppose that $V$ is a semismall perturbation of the operator $P_+^*$ in
$M$. Assume further that Davies' conjecture (Conjecture~\ref{conjD}) holds
true for both $k_{P_0}^M$ and $k_{P_+}^M$ and that \eqref{Ass1m}
holds true. Then
\begin{equation}\label{eqconjMain6a}
\lim_{t\to\infty}\frac{k_{P_+}^M(x,y,t)}{k_{P_0}^M(x,y,t)}=0
\end{equation}
locally uniformly in $M\times M$.
 \end{theorem}
\begin{proof}
The first part of the proof is similar to the corresponding part
in the proof of Theorem~\ref{thm_nonselfadj}. For completeness we
repeat it.

 By Theorem~\ref{thmlpluspos}, we may assume that $\lambda_0(P_+,M)= 0$.

Assume to the contrary that for some $x_0,y_0\in M$ there exists a
sequence $\{t_n\}$ such that $t_n \to \infty$ and
    \begin{equation}\label{eqconjMain2na}
\lim_{n\to\infty}\frac{k_{P_+}^M(x_0,y_0,t_n)}{k_{P_0}^M(x_0,y_0,t_n)}=K>0.
 \end{equation}
Consider the functions $v_+$ and  $v_0$ defined by
\begin{equation}\label{eqconjMain2nab}
v_+(x,t):=\frac{k_{P_+}^M(x,y_0,t)}{k_{P_+}^M(x_0,y_0,t)}\,,\quad v_0(x,t):=\frac{k_{P_0}^M(x,y_0,t)}{k_{P_0}^M(x_0,y_0,t)} \qquad x\in M, t>0.
\end{equation}
By our assumption,
$$\lim_{t\to\infty} v_+(x,t)=u_+(x),\qquad \lim_{t\to\infty} v_0(x,t)=u_0(x),$$
where $u_+\in \mathcal{C}_{P_+}(M)$ and $u_0\in
\mathcal{C}_{P_0}(M)$.

On the other hand, by assumption~\eqref{Ass1m} we have for $t>T(x)$
\begin{equation}\label{gmpssp}
\frac{k_{P_+}^M(x,y_0,t)}{k_{P_0}^M(x,y_0,t)} = \frac{k_{P_+}^M(x,y_1,t)}{k_{P_0}^M(x,y_1,t)}
\times \dfrac{\frac{k_{P_+}^M(x,y_0,t)}{k_{P_+}^M(x,y_1,t)}}{\frac{k_{P_0}^M(x,y_0,t)}{k_{P_0}^M(x,y_1,t)}} \leq C \frac{k_{P_+}^M(x,y_0,t)}{k_{P_+}^M(x,y_1,t)}\times\frac{k_{P_0}^M(x,y_1,t)}{k_{P_0}^M(x,y_0,t)}
\,.
\end{equation}
By our assumption on Davies' conjecture, we have for a fixed $x$
\begin{equation}\label{gmpssp6}
\lim_{t\to\infty}\frac{k_{P_+}^M(x,y_0,t)}{k_{P_+}^M(x,y_1,t)}= \frac{u_+^*(y_0)}{u_+^*(y_1)}\;,\qquad
\lim_{t\to\infty}\frac{k_{P_0}^M(x,y_1,t)}{k_{P_0}^M(x,y_0,t)}=\frac{u_0^*(y_1)}{u_0^*(y_0)},
\end{equation}
where $u_+^*$ and $u_0^*$ are positive solutions of the equation $P_+^*u=0$ and  $P_0^*u=0$ in $M$, respectively. By the elliptic Harnack inequality there exists a positive constant $C_1$ (depending on $P_+^*,P_0^*, y_0,y_1$ but not on $x$) such that
\begin{equation}\label{gmpssp61}
 \frac{u_+^*(y_0)}{u_+^*(y_1)}\leq C_1,\qquad
\frac{u_0^*(y_1)}{u_0^*(y_0)}\leq C_1.
\end{equation}
Therefore,  \eqref{gmpssp} and \eqref{gmpssp61} imply that
\begin{equation}\label{eqm1ssp9}
  \frac{k_{P_+}^M(x,y_0,t_n)}{k_{P_0}^M(x,y_0,t_n)}\leq 2CC_1^2
\end{equation}
for $n$ sufficiently large (which might depend on $x$).

Therefore,
\begin{equation}\label{eqm1ssp}
\frac{k_{P_+}^M(x_0,y_0,t_n)}{k_{P_0}^M(x_0,y_0,t_n)}\times \dfrac{\frac{k_{P_+}^M(x,y_0,t_n)}{k_{P_+}^M(x_0,y_0,t_n)}}
{\frac{k_{P_0}^M(x,y_0,t_n)}{k_{P_0}^M(x_0,y_0,t_n)}}
=
  \frac{k_{P_+}^M(x,y_0,t_n)}{k_{P_0}^M(x,y_0,t_n)}\leq 2CC_1^2
  \,.
\end{equation}
Letting $n\to \infty$ and using~\eqref{eqconjMain2na} and~\eqref{eqconjMain2nab},
we obtain
\begin{equation}\label{eqm2ssp}
Ku_+(x) \leq 2CC_1^2 u_0(x) \qquad x\in M
\,.
\end{equation}
On the other hand, since $V$ is a semismall perturbation of
$P_+^*$ in $M$, Theorem~\ref{thmssp} implies that $u_0(x)\asymp
G_{P_+}^M(x,y_0)$ in $M\setminus \overline{B(y_0,\delta)}$,
with some positive~$\delta$.
Consequently,
\begin{equation}\label{eqm3ssp}
u_+(x) \leq  C_2G_{P_+}^M(x,y_0)\qquad x\in M \setminus \overline{B(y_0,\delta)}
\end{equation}
for some $C_2>0$. In other words, $u_+$ is a global positive solution which has
minimal growth in a neighborhood of infinity in $M$. Therefore
$u_+$ is a ground state of the equation $P_+u=0$ in $M$, but this
contradicts the subcriticality of $P_+$ in $M$.
\end{proof}
%

%%%%%%%%%%%
\section{On the equivalence of heat kernels}\label{secequiv}
In this section we study a general question concerning the equivalence of heat kernels which in turn will give sufficient conditions for the validity of
 the boundedness assumption \eqref{Ass1m} which is assumed in theorems~\ref{thmlpluspos}, \ref{mainthmFKP} and \ref{thm_ssp}.

\begin{definition}\label{defequiv}
Let $P_{i},\,i=1,2$, be two elliptic operators on
$M$ satisfying $\lambda_0(P_i,M) \geq 0$ for $i=1,2 $. We say that the heat kernels
$k_{P_1}^M(x,y,t)$ and $k_{P_2}^M(x,y,t)$ are
{\em equivalent} (respectively, {\em semiequivalent}) if $k_{P_1}^M \asymp
k_{P_2}^M$  on $M \times M \times (0,\infty)$
(respectively, $k_{P_1}^M (\cdot,y_0,\cdot)\asymp k_{P_2}^M(\cdot,y_0,\cdot)$ on
$M\times (0,\infty)$  for some fixed $y_0\in M$).
\end{definition}
There is an intensive literature dealing with (almost optimal) conditions under which two positive (minimal) Green functions are equivalent or semiequivalent
(see \cite{Ancona,Msemismall,P89,P99} and the references therein). On the other hand, sufficient conditions for the equivalence of heat kernels are known only in a few cases (see \cite{LS,MS,Zhang}). In particular, it seems that the answer to the following conjecture is not known.
%%%%%%%%%%%%%
\begin{conjecture}\label{conjequival}
Let $P_1$ and $P_2$ be two subcritical operators of the form \eqref{P} which are
defined on a Riemannian manifold $M$ such that $P_1=P_2$ outside a compact set in $M$.  Then
$k_{P_1}^M$ and $k_{P_2}^M$ are equivalent.
\end{conjecture}
It is well known that certain 3-G inequalities imply the
equivalence of Green functions, and the notions of small and
semismall perturbations is based on this fact.
Moreover, small
(respectively, semismall) perturbations are sufficient conditions and
in some sense also necessary conditions for the equivalence
(respectively, semiequivalence) of the Green functions
\cite{Msemismall,P89,P99}.
We introduce an analog definition for heat kernels (cf. ~\cite{Zhang}).
%%%%%%%%%%%%%%%%%%%%%%%%%%%%%%%%
\begin{definition}\label{def3kineq}
Let $P$ be a subcritical operator in $M$.
We say that a potential $V$ is a {\em $k$-bounded  perturbation}
(respectively, {\em $k$-semibounded perturbation})
with respect to the heat kernel $k_{P}^M(x,y,t)$ if there exists a positive constant $C$ such that the following 3-k inequality is satisfied:
\begin{equation}\label{eq3kineq}
\int_0^t\int_M k_{P}^M(x,z,t-s)|V(z)|k_{P}^M(z,y,s)\,\mathrm{d}\mu(z)\,\mathrm{d}s \leq C k_{P}^M(x,y,t)
\end{equation}
for all $x,y\in M$ (respectively,  for a fixed $y\in M$, and all $x\in M$) and $t>0$.
\end{definition}
The following result shows that the notion of $k$-(semi)boundedness
is naturally related to the (semi)equivalence of heat kernels.
\begin{theorem}\label{thmbounded}
Let $P$ be a subcritical operator in $M$,
and assume that the potential $V$ is $k$-bounded  perturbation
(respectively, $k$-semibounded perturbation)
with respect to the heat kernel $k_{P}^M(x,y,t)$.
Then there exists $c>0$ such that for any $|\varepsilon|<c$ the heat kernels
$k_{P+\varepsilon V}^M(x,y,t)$ and $k_{P}^M(x,y,t)$ are
equivalent (respectively, semiequivalent).
\end{theorem}
\begin{proof}
Consider the iterated kernels
$$
k_{P}^{(j)}(x,y,t):= \left\{
  \begin{array}{ll}
    k_{P}^M(x,y,t) & j=0, \\[4mm]
    \int_0^t\!\!\int_M \!k_P^{(j-1)}(x,z,t-s)  V(z)
  k_{P}^M(z,y,s)
  \,\mathrm{d}\mu(z)\,\mathrm{d}s &  j\geq 1.
  \end{array}
\right.
$$
Using~\eqref{eq3kineq} and an induction argument, it follows that
\begin{multline*}
  \sum_{j=0}^\infty |\varepsilon|^j|k_{P}^{(j+1)}(x,y,t)|\\
  \leq
  \left(1+C|\varepsilon|+C^2|\varepsilon|^2+\dots\right) k_{P}^M(x,y,t)
  = \frac{1}{1-C|\varepsilon|} \, k_{P}^M(x,y,t)
\end{multline*}
provided that $|\varepsilon|<C^{-1}$. Consequently, for such $\varepsilon$ the Neumann series
$$\sum_{j=0}^\infty (-\varepsilon)^j k_{P}^{(j+1)}(x,y,t)$$ converges locally uniformly in $M\times M \times \mathbb{R}_+$ to $k_{P+\varepsilon V}^M(x,y,t)$ which in turn implies that Duhamel's formula
\begin{equation}\label{Duhamel}
  k_{P+\varepsilon V}^M(x,y,t) = k_{P}^M(x,y,t)
  - \varepsilon \!\! \int_0^t\!\!\int_M \!\!k_{P}^M(x,z,t-s)  V(z)
  k_{P+\varepsilon V}^M(z,y,s)
  \,\mathrm{d}\mu(z)\,\mathrm{d}s
\end{equation}
is valid. Moreover, we have
$$k_{P+\varepsilon V}^M(x,y,t) \leq \frac{1}{1-C|\varepsilon|} \, k_{P}^M(x,y,t).$$
The lower bound
$$C_1k_{P}^M(x,y,t) \leq k_{P+\varepsilon V}^M(x,y,t)$$
(for $|\varepsilon|$ small enough) follows from the upper bound using \eqref{Duhamel} and \eqref{eq3kineq}.
\end{proof}
\begin{corollary}\label{cor7}
Assume that $P$ and $V$ satisfy the conditions of Theorem~\ref{thmbounded}, and suppose
further that $V$ is nonnegative. Then there exists $c>0$ such that for any $\varepsilon>-c $ the heat kernels
$k_{P+\varepsilon V}^M(x,y,t)$ and $k_{P}^M(x,y,t)$ are
equivalent (respectively, semiequivalent).
\end{corollary}
\begin{proof}
By Theorem~\ref{thmbounded} the heat kernels
$k_{P+\varepsilon V}^M(x,y,t)$ and $k_{P}^M(x,y,t)$
are equivalent (respectively, semiequivalent) for any
$|\varepsilon|<c$. Recall that by the generalized maximum principle,
$$k_{P+\varepsilon V}^M(x,y,t)\leq k_{P}^M(x,y,t)\qquad \forall \varepsilon>0.$$
On the other hand, using also Lemma~\ref{lem_convex}, we obtain the desired equivalence also for $\varepsilon\geq c$.
\end{proof}
%%%%%%%%%%%%%%%%%%%
\begin{theorem}\label{thmcond1}
Let $P_0$ be a critical operator in $M$. Assume that $V=V_+ - V_-$ is a potential such that $V_\pm \geq 0$ and  $P_+:=P_0+V$ is subcritical in $M$.

Assume further that $V_-$ is $k$-semibounded perturbation with respect to the heat kernel $k_{P_+}^M(x,y_1,t)$. Then condition \eqref{Ass1m} is satisfied uniformly in $x$. That is, there exist positive constants $C$ and $T$ such that
\begin{equation}\label{Ass1n}
    k_{P_+}^M(x,y_1,t)\leq C k_{P_0}^M(x,y_1,t)\qquad \forall x\in M, t>T.
\end{equation}
  \end{theorem}
%%%%%%%%%%%%%%%%%%%%%%%%
\begin{proof}
 By Corollary~\ref{cor7}, the heat kernels $k_{P_+}^M(x,y_1,t)$ and $k_{P_+ + V_-}^M(x,y_1,t)$
 are semi\-equivalent. Note that  $P_+ + V_-=P_0 + V_+$, Therefore we have
\begin{equation}\label{eq27}
C^{-1} k_{P_+}^M(x,y_1,t) \leq k_{P_0+V_+}^M(x,y_1,t)\leq k_{P_0}^M(x,y_1,t) \qquad \forall x\in M, t>0.
\end{equation}
\end{proof}

%%%%%%%%%%%%%%%%%%%%%%%%%%%%%
\section*{Acknowledgments}
Part of this research was done while D.~K. was visiting the
Technion. D.~K. would like to thank the Technion for the kind
hospitality. The work of D.~K. was
partially supported by the Czech Ministry of Education,
Youth and Sports within the project LC06002. M.~F. and Y.~P. acknowledge the support of the Israel Science
Foundation (grant No. 587/07) founded by the Israeli Academy of
Sciences and Humanities. M.~F. acknowledges also the support of the Israel Science
Foundation (grant No. 419/07) founded by the Israeli Academy of
Sciences and Humanities. M.~F. was also partially supported by a fellowship
of the UNESCO fund.

 \begin{remark}[Added on 17.5.2010]\label{Remarkn}
After the paper has been published, we realized that the generalized principal eigenvalue is characterized by the following formula
  \begin{equation}\label{logform}
\lim_{t \to \infty} \frac{\log k_P^M(x,y,t)}{t} =-\lambda_0(P,M) .
\end{equation}

The above characterization of $\lambda_0$ is well-known in the symmetric case, 
see for example \cite[Theorem~10.24]{Grigoryan}.
The needed upper bound for the validity of \eqref{logform} 
for general elliptic operators of the form \eqref{P} 
follows directly from Theorem~\ref{mainthmhk}, 
while the lower bound follows from the large time behavior of the heat kernel 
in a smooth bounded domain using a standard exhaustion argument 
(cf.~the proof of  \cite[Theorem~10.24]{Grigoryan}).

Consequently, if $P_0$ is a critical operator in $M$, 
and $P_+$ is a subcritical operator in $M$ 
satisfying $\lambda_+:=\lambda_0(P_+,M)> 0$, 
then Conjecture~\ref{conjMain} holds true without any further assumption 
(cf.~Theorem~\ref{thmlpluspos}).
\end{remark}

%%%%%%%%%%%%%%%%%%%%%%%%%%%%%%%%


\begin{thebibliography}{99}

\bibitem{Agmon82}
\newblock S.~Agmon,
\newblock \emph{On positivity and decay of solutions of second order
elliptic equations on Riemannian manifolds},
\newblock in ``Methods of
Functional Analysis and Theory of Elliptic Equations" (ed. D.~Greco), Liguori, Naples, (1983), 19--52.
\MR{0819005}
%
\bibitem{Ancona}
\newblock A.~Ancona,
\newblock \emph{First eigenvalues and comparison of Green's functions
for elliptic operators on manifolds or domains},
\newblock J.~Anal. Math., {\bf 72} (1997), 45--92.
\MR{1482989}
%
\bibitem{ABJ}
\newblock J-P~Anker, P.~Bougerol and T.~Jeulin,
\newblock \emph{The infinite
Brownian loop on symmetric space},
\newblock Rev. Mat. Iberoamericana, {\bf 18} (2002), 41--97.
\MR{1924687}
%
\bibitem{BCG}
\newblock M.~Barlow, T.~Coulhon and A.~Grigor'yan,
\newblock \emph{Manifolds and graphs with slow
heat kernel decay},
\newblock Invent. Math., {\bf 144} (2001), 609--649.
\MR{1833895}
%
\bibitem{CMS}
\newblock P.~Collet, S.~Mart\'\i nez and J.~San Mart\'\i n,
\newblock \emph{Ratio limit
theorems for a Brownian motion killed at the boundary of a
Benedicks domain},
\newblock  Ann. Probab., {\bf 27} (1999), 1160--1182.
\MR{1733144}
%
\bibitem{D}
\newblock E.~B.~Davies,
\newblock \emph{Non-Gaussian aspects of heat kernel behaviour},
\newblock J. London Math. Soc. (2), {\bf 55} (1997), 105--125.
\MR{1423289}
%
\bibitem{Grigoryan}
\newblock A.~Grigor'yan,
\newblock ``Heat Kernel and Analysis on Manifolds",
\newblock AMS/IP Studies in Advanced Mathematics, 47. American Mathematical Society, Providence, RI; International Press, Boston, MA, 2009.
\MR{MR2569498}
%
\bibitem{Kozma} G.~Kozma, private communication, and a talk in a probability
seminar at UBC, 2006.
%
\bibitem{KZ}
\newblock D.~Krej\v{c}i\v{r}\'{\i}k and E.~Zuazua,
\newblock \emph{The Hardy inequality and the
heat equation in twisted tubes}, to appear in  J.~Math. Pures Appl. (\url{http://dx.doi.org/10.1016/j.matpur.2010.02.006}).
%
\bibitem{LS}
\newblock V.~Liskevich and Y.~Semenov,
\newblock \emph{Two-sided estimates of the heat kernel of the Schr\"odinger operator},
\newblock Bull. London Math. Soc., {\bf 30} (1998), 596--602.
\MR{1642818}
%
\bibitem{MS}
\newblock P.~D.~Milman and Yu.~A.~Semenov,
\newblock \emph{Heat kernel bounds and desingularizing weights},
\newblock J.~Funct. Anal.,  {\bf 202}  (2003),  1--24.
\MR{1994762}
%
\bibitem{M_large_time}
\newblock M.~Murata,
\newblock \emph{Positive solutions and large time behaviors of Schr\"odinger semigroups, Simon's problem},
\newblock J.~Funct. Anal., {\bf 56}  (1984), 300--310.
\MR{0743843}
%
\bibitem{Msemismall}
\newblock M.~Murata,
\newblock \emph{Semismall perturbations in the Martin theory for elliptic equations},
\newblock Israel J. Math.,  {\bf 102}  (1997), 29--60.
\MR{1489100}
%
\bibitem{P89}
\newblock Y.~Pinchover,
\newblock \emph{Criticality and ground states for second-order elliptic equations},
\newblock J.~Differential Equations, {\bf 80} (1989), 237--250.
\MR{1011149}
%
\bibitem{P90}
\newblock Y.~Pinchover,
\newblock \emph{On criticality and ground states of second order elliptic
equations, II},
\newblock J. Differential Equations, {\bf 87} (1990), 353--364.
\MR{1072906}
%
\bibitem{Pheat}
\newblock Y.~Pinchover,
\newblock \emph{Large time behavior of the heat kernel
and the behavior of the Green function near criticality for
nonsymmetric elliptic operators},
\newblock J. Funct. Anal., {\bf 104} (1992), 54--70.
\MR{1152459}
%
\bibitem{P99}
\newblock Y.~Pinchover,
\newblock \emph{Maximum and anti-maximum principles and eigenfunctions
estimates via perturbation theory of positive
solutions of elliptic equations},
\newblock  Math. Ann., {\bf 314}  (1999),  555--590.
\MR{1704549}
%
\bibitem{P03}
\newblock Y.~Pinchover,
\newblock \emph{Large time behavior of the heat kernel},
\newblock J. Funct. Anal., {\bf 206} (2004), 191--209.
\MR{2024351}
%
\bibitem{PDavies}
\newblock Y.~Pinchover,
\newblock \emph{On Davies' conjecture and strong ratio
limit properties for the heat kernel},
\newblock  in ``Potential Theory in
Matsue", Proceedings of the International Workshop on Potential
Theory, 2004 (eds. H.~Aikawa, et al.), Advanced Studies in Pure
Mathematics {\bf 44}, Mathematical Society of Japan, Tokyo, (2006),
339--352.
\MR{2279767}
%
\bibitem{Pliouv}
\newblock Y.~Pinchover,
\newblock \emph{A Liouville-type theorem for Schr\"odinger operators},
\newblock Comm. Math. Phys., {\bf 272} (2007), 75--84.
\MR{2291802}
%
\bibitem{pinsky}
\newblock R.~G.~Pinsky,
\newblock ``Positive Harmonic Function and Diffusion",
\newblock Cambridge University Press, Cambridge, 1995.
\MR{1326606}
%
\bibitem{Simon82}
\newblock B.~Simon,
\newblock \emph{Schr\"odinger semigroups},
\newblock Bull. Amer. Math. Soc. (N.S.), {\bf 7}  (1982), 447--526.
\MR{0670130}
%
\bibitem{Zhang}
\newblock Q.~S.~Zhang,
\newblock \emph{A sharp comparison result concerning Schr\"odinger heat kernels},
\newblock Bull. London Math. Soc., {\bf 35}  (2003),  461--472.
\MR{1978999}
%
%
\end{thebibliography}
\end{document}